\documentclass[11pt, reqno]{amsart}

\usepackage{amssymb}
\usepackage{hyperref} 
\usepackage{mathrsfs,bbold}
\usepackage{amsmath}
\usepackage{amsthm}
\usepackage{enumerate}
\usepackage{dsfont}
\usepackage{color}
\usepackage[a4paper]{geometry}
\usepackage[utf8]{inputenc}
\usepackage{stmaryrd}
\usepackage{titlesec}
\usepackage{mathabx}
\usepackage{appendix}
\usepackage{todonotes, fancyhdr}
\usepackage{chngcntr}

\makeatletter
   \def\overbracket#1{\mathop{\vbox{\ialign{##\crcr\noalign{\kern3\p@}
         \downbracketfill\crcr\noalign{\kern3\p@\nointerlineskip}
         $\hfil\displaystyle{#1}\hfil$\crcr}}}\limits}
   \def\underbracket#1{\mathop{\vtop{\ialign{##\crcr
         $\hfil\displaystyle{#1}\hfil$\crcr\noalign{\kern3\p@\nointerlineskip}
         \upbracketfill\crcr\noalign{\kern3\p@}}}}\limits}
   \def\overparenthesis#1{\mathop{\vbox{\ialign{##\crcr\noalign{\kern3\p@}
         \downparenthfill\crcr\noalign{\kern3\p@\nointerlineskip}
         $\hfil\displaystyle{#1}\hfil$\crcr}}}\limits}
   \def\underparenthesis#1{\mathop{\vtop{\ialign{##\crcr
         $\hfil\displaystyle{#1}\hfil$\crcr\noalign{\kern3\p@\nointerlineskip}
         \upparenthfill\crcr\noalign{\kern3\p@}}}}\limits}
   \def\downparenthfill{$\m@th\braceld\leaders\vrule\hfill\bracerd$}
   \def\upparenthfill{$\m@th\bracelu\leaders\vrule\hfill\braceru$}
   \def\upbracketfill{$\m@th\makesm@sh{\llap{\vrule\@height3\p@\@width.7\p@}}%
     \leaders\vrule\@height.7\p@\hfill
     \makesm@sh{\rlap{\vrule\@height3\p@\@width.7\p@}}$}
   \def\downbracketfill{$\m@th
     \makesm@sh{\llap{\vrule\@height.7\p@\@depth2.3\p@\@width.7\p@}}%
     \leaders\vrule\@height.7\p@\hfill
     \makesm@sh{\rlap{\vrule\@height.7\p@\@depth2.3\p@\@width.7\p@}}$}
   \makeatother

\usepackage[all]{xy}
\let\oldtocsection=\tocsection
\let\oldtocsubsection=\tocsubsection
\let\oldtocsubsubsection=\tocsubsubsection

\renewcommand{\tocsection}[2]{\hspace{0em}\oldtocsection{#1}{#2}}
\renewcommand{\tocsubsection}[2]{\hspace{1em}\oldtocsubsection{#1}{#2}}
\renewcommand{\tocsubsubsection}[2]{\hspace{2em}\oldtocsubsubsection{#1}{#2}}

\numberwithin{subsection}{section}
\numberwithin{subsubsection}{subsection}
\numberwithin{theorem}{section}
\numberwithin{equation}{section} 

\newcommand{\ssk}{\smallskip}

\renewcommand{\epsilon}{\varepsilon}

\newcommand\bbR{\mathbb{R}}

\titleformat{\section}[block]
{\filcenter\normalfont\sffamily\bfseries\Large}
{{\hspace{-0.7cm}}\thesection \hspace{0.2em} --\vspace{0.3cm}}{0.5em}{}

\titleformat{\subsection}[block]
{\filcenter\normalfont\sffamily\bfseries\large}  						  
{\hspace{-0.7cm}\thesubsection \hspace{0.5em}--\vspace{0.3cm}}{.5em}{}  
\titlespacing{\subsection}{-0pc}{1.5ex plus .1ex minus .2ex}{0pc}

\titleformat{\subsubsection}[block]
{\normalfont\sffamily\bfseries}{\hspace{-1cm}\thesubsubsection}{.5em}{}
\titlespacing{\subsubsection}{15pc}{1ex plus .1ex minus .2ex}{1pc}

\newtheoremstyle{mystyle}
{3pt}               
{3pt}               
{\it }                      
{}                      
{\sffamily\bfseries}             
{}                      
{0.5em}                 
{\llap{#2. }#1{$\;$ --}}

\theoremstyle{mystyle}

\newtheorem{thm}{Theorem}[section]
		\newtheorem*{thm*}{Theorem}
\newtheorem{cor}[thm]{\hspace{-0.13cm} {Corollary}}  
\newtheorem{lem}[thm]{\hspace{-0.15cm} {Lemma}}  
\newtheorem{prop}[thm]{\hspace{-0.16cm} {Proposition}}   

\newtheoremstyle{mystyle2}
{3pt}               
{3pt}               
{\it }                      
{}                      
{\sffamily\bfseries}             
{}                      
{0.5em}                 
{\llap{#2 }#1{\hspace{0.2cm}--}}
\theoremstyle{mystyle2}

\newtheorem*{definition*}{Definition}
\newtheorem*{theorem*}{Theorem}
\newtheorem*{Remark*}{Remark}
\newtheorem*{lem*} {Lemma}
\newtheorem*{defn*} {Definition}
\newtheorem*{prop*} {Proposition}
\newtheorem*{cor*} {Corollary}

\newcommand\restr[2]{{
  \left.\kern-vlldelimiterspace 
  #1 
  \vphantom{\big|} 
  \right|_{#2} 
  }}

\newcommand{\R}{\mathbb{R}}

\newcommand{\rX}{\mathbf{X}}
\newcommand{\XX}{\mathbb{X}}

\newcommand{\norm}[1]{\left\|#1\right\|}
\newcommand{\normhh}[1]{\|#1\|_{2\alpha}}

\newcommand{\normhb}[1]{\|#1\|_{\beta}}
\newcommand{\normhhb}[1]{\|#1\|_{2\beta}}

\newcommand{\normh}[1]{\|#1\|_{\alpha}}
\newcommand{\normsup}[1]{\|#1\|_{\infty}}
\newcommand{\CO}{C^{0}([0,T],\R^d)}

\newcommand\abs[1]{| #1 |}
\newcommand{\normDX}[1]{\left\|#1\right\|_{\alpha,2\alpha}}

\newcommand{\CD}{C_b^{1,0}}

\newcommand{\normCD}[1]{\|#1\|_{\CD}}

\begin{document}

\begin{center}
{\Huge\sffamily{Young and rough differential inclusions}   \vspace{0.5cm}}
\end{center}
\vskip 5ex minus 1ex

\begin{center}
{\sf I. BAILLEUL, A. BRAULT and L. COUTIN}
\end{center}

\vspace{1cm}

\begin{center}
\begin{minipage}{0.9\textwidth}
  \renewcommand\baselinestretch{0.7} 
  {\small \textbf{\textsf{\noindent Abstract.}} \textit{We define in this work a notion of Young differential inclusion 
$$
dz_t \in F(z_t)dx_t,
$$
for an $\alpha$-H\"older control $x$, with $\alpha>1/2$, and give an existence result for such a differential system. As a by-product of our proof, we show that a bounded, compact-valued, $\gamma$-H\"older continuous set-valued map on the interval $[0,1]$ has a selection with finite $p$-variation, for $p>1/\gamma$. We also give a notion of solution to the rough differential inclusion 
$$
dz_t \in F(z_t)dt + G(z_t)d{\bf X}_t,
$$
for an $\alpha$-Hölder rough path $\bf X$ with $\alpha\in
\left(\frac{1}{3},\frac{1}{2}\right]$, a set-valued map $F$ and  a single-valued one form $G$. Then, we prove the existence of a solution to the inclusion when $F$ is bounded and lower semi-continuous with compact values, or upper semi-continuous with compact and convex values.}}
\end{minipage}
\end{center}

\vspace{1cm}

\section{Introduction}
\label{SectionIntro}
\subsection{Setting}
One of the motivations for considering differential inclusions comes from the study of differential equations with discontinuous coefficients. In the setting of a possibly time-dependent widely discontinuous vector field on $\mathbb{R}^d$, it makes sense to replace the dynamical prescription 
\begin{align*}
\dot{z}_t = f(t,z_t),
\end{align*}
by 
\begin{align*}
\dot{z}_t \in F(t,z_t),
\end{align*}
where $F(t,z)$ is here the closed set of cluster points ot $F(s,w)$, as $(s,w)$ converges to $(t,z)$. This somehow accounts for the impossiblity to make measurements with absolute precision. It also makes sense to take for $F(t,z)$ the convex hull of the former set. A set-valued application $F$ is a map from $[0,T] \times {\mathbb R}^d $ into the power set of ${\mathbb R}^d$. Different natural set-valued extensions of $f$ may have different regularity properties; in any case, switching from the differential equation prescription to the differential inclusion formulation somehow allows to account for the uncertainty in the modeling. A Caratheodory solution of the differential inclusion 
\begin{align}\label{intro-1}
\dot{z}_t \in F(t,z_t),~~t\in [0,T],~~z_0 = \xi \in {\mathbb R}^d
\end{align}
is an absolutely continuous path $z$ started from $\xi$, whose derivative $\dot z$ satisfies
\begin{align*}
\dot{z}_t \in F(t,z_t)
\end{align*}
at almost all times $t\in [0,T]$. Existence of solutions of
differential inclusions and their properties were widely studied; see
for instance J.P. Aubin and A. Cellina's book \cite{AubinCellina} for
an authoritative pedagogical treatment. Equation \eqref{intro-1} has
at least one solution under two  kinds of assumptions on the set
valued map $F$: \textcolor{black}{either $F$ is upper semicontinuous
  with compact convex values, or $F$ is bounded and lower
  semicontinuous with compact values.} Recall that a set-valued function $F$ on  $[0,T]\times\mathbb{R}^d$ is said to be \textit{upper semicontinuous} if one can associate to every spacetime point $(s,w)$ an open neighbourhood $U$ of $(s,w)$ and an open neighbourhood $V$ of $F(s,w)$ such that $F(U)\subset V$, and $F$ is said to be \textit{lower semicontinuous} if for any $(s,w)$, any $w'\in F(s,w)$ and any neighbourhood $V(w')$ of $w'$, there exists a neighbourhood $U(s,w)$ of $(s,w)$ such that $F(t,x)\cap V(w')\neq\emptyset$, for all $(t,x)\in U(s,w)$. The existence proofs for solutions of the differential inclusion \eqref{intro-1} generally rely on either approximation schemes or a fixed point reformulation.

\smallskip

Differential equations, and their extensions into differential inclusions, are not the only kind of natural dynamics. Some extensions to stochastic cases where also studied, after the pioneering works of Aubin-da Prato \cite{AD90, AD95, AD98, ADF2000} and Kisielewicz \cite{Kisielwicz1,Kisielwicz2,Kisielwicz3}. The works  \cite{AD90, AD95, AD98, ADF2000} were essentially motivated by studying viability questions in a stochastic setting. M. Kisielewicz in \cite{Kisielwicz1} defined the notion of solution and obtained the existence of stochastic differential inclusions of the form 
\begin{align}\label{intro-2}
X_t -X_s \in \int_s^t F(r,X_r) dr + \int_s^t G(r,X_r) d W_r
\end{align}
where $W$ is an $\bbR^\ell$-valued Brownian motion, $F$ is a lower semicontinuous set-valued random map with values in ${\mathbb R}^d$ and $G$ is a lower semicontinuous set-valued map with values in $L({\mathbb R}^d,\mathbb{R}^\ell)$. In \cite{Kisielwicz2}, M. Kisielewicz also studied the case where equation \eqref{intro-2} also contains an additional compound Poisson random measure term; the case of semimartingale drivers was also investigated. All proofs fundamentally rely on the isometry property of stochastic integration with respect to Brownian motion or compound Poisson measures. There are however a number of interesting non-semimartingale random processes of practical relevance, such as Mandelbrot's fractional Brownian motion \cite{Mandelbrot68} or random Fourier series \cite{FrizGuliRiedel}. Their sample paths are  $\alpha$-H\"older continuous, for  a regularity exponent $\alpha \in (0,1).$ One can use T. Lyons' work \cite{Lyons94} to solve  differential equations driven by a fractional Brownian motion, for $\alpha >1/2$; it relies on the notion of Young integral \cite{Young}. The case $1/4\leq \alpha \leq 1/2$ is much more involved and can be handled using Lyons' theory of rough paths \cite{Lyons98}. In \cite{Levakov15}, A. Levakov and M. Vas'kovskii obtain the existence of solutions to stochastic differential inclusions of the form
\begin{align*}
X_t -X_s \in \int_s^t F(r,X_r) dr + \int_s^t G_1(r,X_r) d W_r+ \int_s^t G_2(r,X_r) d W_r^{\textsc{FBM}}
\end{align*}
where $W^{\textsc{FBM}}$ is a fractional Brownian motion with $\alpha$-H\"older continuous sample paths, with $\alpha >1/2$, and the set valued map $G_2$ is a globally Lipschitz function of $x$ that takes values in the set of nonempty compact \textit{convex} subsets of $\mathbb{R}^d$ and satisfies a local $\delta$-H\"older condition as a function of time, with $\delta >1-\alpha$.

\subsection{Young differential inclusions} 
Motivated by control problems, the first aim of this article is to define and prove the existence of solutions of Young differential inclusions 
\begin{align}\label{intro4}
z_t= \xi + \int_0^t v_r \,dx_r,\quad
v_r\in  F(z_r),
\end{align}
where $x$ is an ${\mathbb R}^\ell$-valued  $\alpha$-H\"older continuous control, with $\alpha >1/2$, and $F$ is a  $\gamma$ -H\"older continuous set valued map with compact values, for a regularity exponent $\gamma\in\left(\frac{1}{\alpha} -1,1\right)$. {We do not require that $F$ takes values in convex sets}. The notion of a solution to a Young differential inclusion involves
a number of elementary results on Young integrals that are recalled in
Appendix \ref{SectionAppendix}.\\

\noindent \emph{Notations.} We gather here a number of notations that are used throughout the work.   \vspace{0.15cm}

\begin{itemize}
   \item For any integer $m$, a point $a\in \R^{m}$ and a positive real number $R$, we denote by $B(a,R)$ the closed ball of center $a$ and radius $R$ in $\R^{m}$.   \vspace{0.1cm}
   \item For $\alpha\in (0,1]$ and $U,V$ Banach spaces, we denote by $C^\alpha(U,V)$ the space of $\alpha$-Hölder continuous functions from $U$ to $V$. We write $\norm{x}_\alpha$ for the $\alpha$-H\"older semi-norm of a path $x\in C^\alpha(U,V)$. We set $x_{s,t}:=x_t-x_s$, for $s,t\in U$.   \vspace{0.1cm} 
   \item Given a set $S$, we denote by $2^S$ the power set of $S$,  $\dot 2^S$ for $2^S\backslash\{\emptyset\}$ and $\mathcal{K}\big(S\big)$ the compact sets of $\dot 2^S$. We denote by $L(\bbR^\ell,\bbR^d)$ the space of linear maps from $\R^\ell$ to $\R^d$.
Recall that endowing the space $2^{L(\bbR^\ell,\bbR^d)}$ with the Hausdorff pseudo-metric turns the space $\mathcal{K}\big(L(\bbR^\ell,\bbR^d)\big)$ into a complete metric space. For $F:\R^d\rightarrow \mathcal{K}(L(\R^\ell,\R^d))$ a $\gamma$-Hölder
set-valued map, we denote again by $\norm{F}_\gamma$ the $\gamma$-Hölder semi-norm of
$F$.  (See Chapter 1, Section 5 in \cite{AubinCellina})\vspace{0.1cm}
 
   \item If $f$ is a single-valued map of $p$-variation, we denote by $\norm{f}_{p-\textrm{\emph{var}}}$ the $p$-variation semi-norm of $f$  -- more on this notation in Appendix \ref{SectionAppendix}.
\end{itemize}

\begin{defn*}
Let $x$ be an element of $C^{\alpha}([0,T],\R^\ell)$, with $\alpha\in \left(\frac{1}{2},1\right]$, and  $F : \bbR ^d\mapsto \dot 2^{L(\bbR^\ell,\bbR^d)}$ be a set-valued map. A solution to the \emph{Young differential inclusion}
\begin{equation}
\label{EqYDI}
dz_t \in F(z_t)dx_t,\quad z_0=\xi\in\R^d
\end{equation}
is a pair of paths $(z,v)$, defined on the time interval $[0,T]$,   \vspace{0.15cm}

\begin{itemize}
   \item with $v$ a $L(\R^\ell,\R^d)$-valued path of finite $p$-variation
     such that  $\alpha+\frac{1}{p}>1$,   \vspace{0.1cm}
   
   \item and for all  $0\leq t\leq T$,
     \begin{align*}
  v_t\in F(z_t)\quad \textrm{and}\quad      z_t = \xi+\int_0^tv_s\,dx_s.
     \end{align*}
\end{itemize}
\end{defn*}

The letter $v$ is chosen for "velocity". 
The integral $\int_0^tv_s\,dx_s$ makes sense as a Young integral under the assumption $\alpha+\frac{1}{q}>1$ -- see Appendix  ~\ref{SectionAppendix}.

\begin{thm}
\label{ThmMain}
Let a positive time horizon $T$ be given and $x\in C^\alpha([0,T],\R^\ell)$ with
$\alpha\in \left(\frac{1}{2},1\right]$. Let $F : \bbR ^d\mapsto
\mathcal{K}\big(L(\bbR^\ell,\bbR^d)\big)$ be a bounded set-valued map  with nonempty compact values and $\gamma$-H\"older, for a regularity exponent $\gamma\in\big(\frac{1}{\alpha}-1,1\big)$. Then, for any initial condition $\xi\in\bbR^d$, the Young differential inclusion 
\begin{align*}
dz_t \in F(z_t)dx_t,
\end{align*}
has a solution path started from $\xi$, defined over the time interval $[0,T]$.
\end{thm}

For a single-valued map, this is a consequence of Young's original
result \cite{Young, Lyons94}.   Unlike the setting of ordinary
differential equations, no uniqueness is to be expected in the
present setting. Our regularity condition on the set-valued map $F$ is the same as in the ordinary differential setting. However, there may be no H\"older, or even continuous, selection of $F$ -- see Proposition 8.2 in \cite{Chistyakov}, so existence results for Young differential inclusions do \textit{not} follow from existence results for Young differential equations. We refer the reader to Chapters 1 and 2 of \cite{AubinCellina} for the basics on differential inclusions. Note that since continuous paths with finite $p$-variation can be reparamatrized into $1/p$-H\"older paths, the result of Theorem \ref{ThmMain} holds for continuous paths $x$  with finite $p$-variation $ 1\leq p<2.$

\subsection{A selection result}
As a by-product of our proof of Theorem~\ref{ThmMain}, we show in Section \ref{SectionSideSelection} that a bounded, nonempty compact-valued, $\gamma$-H\"older continuous set-valued map on the interval $[0,1]$ has a selection with finite $p$-variation, for $p>1/\gamma$.

\begin{thm} 
\label{ThmSideSelection}
Pick $\gamma\in(0,1]$, and let $F: [0,1] \mapsto \mathcal{K}({\mathbb R}^{d})$, be a bounded, $\gamma$-H\"older nonempty compact set-valued map. Then, for any $p > 1/\gamma$ and any $\xi \in F(0)$, there exists a map $f: [0,1] \mapsto {\mathbb R}^d$, of finite $p$-variation, such that $f(t)\in F(t)$, for all $0\leq t\leq 1$, with \textcolor{black}{$f(0)=\xi$,} and which furthermore satisfies the estimate \textcolor{black}{for a constant $C_{\gamma,p}$ depending only on $\gamma,p$}
\begin{align*}
\|f\|_{p-\textrm{\emph{var}}} \leq \textcolor{black}{C_{\gamma,p}}\|F\|_{\gamma}.
\end{align*}
\end{thm}

This statement partly answers an open question in Chistyakov and Galkin's work \cite{Chistyakov} -- see Remark 8.1 therein.

\subsection{Rough differential inclusions} 
Stochastic analysis has undergone a real change under the impulse of T. Lyons' theory of rough paths \cite{Lyons98}. It provides a deep understanding of stochastic differential equations and disentangles completely in that setting probabilistic and dynamical matters. As a first step to establishing a full theory of rough differential inclusions, the second aim of the present work is to deal with \textit{rough differential equations perturbed by set-valued drifts}. Fix a finite positive time horizon $T$. We refer here the reader to Appendix \ref{SectionAppendix} for basics on rough paths and rough integrals, in the setting of controlled paths; it suffices here to say that controlled paths are needed to make sense of the rough integral that appears in the next definition.   \vspace{0.1cm}

  For an integer $k\geq 1$ and $\gamma\in [0,1]$, let us
  denote by $C_b^{k,\gamma}$ the space of bounded maps $G$ from $\R^d$ to
  $L(\R^\ell,\R^d)$ which are $k$-times differentiable, with bounded derivatives, and have a
  $\gamma$-H\"older $k$-th derivative. We endow $C_b^{k,\gamma}$ with the norm
  \begin{align*}
    \norm{G}_{C_b^{k,\gamma}}:=\sum_{i=0}^k \normsup{D^iG}+\norm{G}_\gamma\mathbf{1}_{\gamma>0},
  \end{align*}
where $D^iG$ is the $i$-th derivative of $G$. By convention $D^0G:=G$.

\begin{defn*}   
Let $T$ be a fixed positive time horizon. Pick $\xi\in\R^d$,  let $\Omega$ be a  subset of
$\R\times\R^d$ containing $[0,T]\times \{\xi\}$.  Let $\bf X$ be a
weak geometric $\alpha$-H\"older rough path, with  $1/3<\alpha\leq
1/2$. Let $F:\Omega\rightarrow\dot 2^{\R^d}$, be a set-valued map, and
{\color{black}$G\in C_b^{1,\gamma}$ with $(2+\gamma)\alpha>1$}. An $\R^d$-valued path $z$ started from $\xi$ is said to be a \emph{solution to the rough differential inclusion}
\begin{align}   \label{EqRDI}
dz_t \in F(t,z_t)dt + G(z_t)d{\bf X}_t,
\end{align}
if $z$ is part of a path $(z,z')$ controlled by $\bf X$, and there
exists an absolutely continuous path $x$ from $[0,T]$ to $\R^d$
starting from zero, such that 
\begin{align*}
\dot{x}_t\in F(t,z_t)
\end{align*}
at Lebesgue almost all times $0\leq t\leq T$, and one has
\begin{align*}
z_t = {\color{black}\xi} + x_t + \int_0^t G(z_s)d{\bf X}_s
\end{align*}
at all times.
\end{defn*}

The regularity assumption on $G$ is the optimal classical regularity assumption required to have a well-posed rough differential equation when $F$ is null. The technics used to prove the existence of a solution to equation \eqref{EqRDI} are different depending on the regularity of the set-valued drift $F$. We first deal with the case where the drift is upper semicontinuous. We need in that case to assume that $F$ takes values in the set of \textit{convex} compact subsets of $\mathbb{R}^d$.

\begin{thm}   \label{ThmRDIUSD}
Let $\bf X$ be a weak geometric $\alpha$-H\"older rough path, with $1/3<\alpha\leq 1/2$, and let $G$ be a $C_b^{2,\gamma}$ one form, with $\alpha(2+\gamma)>1$. Let $F$ be an upper semicontinuous set-valued drift. Assume further that $F$ is locally bounded, and takes it values in the set of nonempty compact convex subsets of $\R^d$. Then there exists a time horizon  $T^*\in (0,T)$ and a solution $z$ to the rough differential inclusion \eqref{EqRDI}, defined on the time interval $[0,T^*]$.
\end{thm}

In contrast to the preceding statement, we do not need to assume that $F$ has convex images in the case where it is lower semicontinuous; we assume instead a mild boundedness assumption on $F$.

\begin{thm}   \label{ThmRDI}
Let $\bf X$ be a weak geometric $\alpha$-H\"older rough path, with $1/3<\alpha\leq 1/2$, and let $G$ be a $C_b^{1,\gamma}$ one form, with $\alpha(2+\gamma)>1$. Assume that the set-valued drift $F$ is defined on a closed subset $\Omega$ of $\R\times\R^d$, where it is lower semicontinuous. Assume further that $F$ takes values in the set $\mathcal{K}(\R^d)$ of nonempty compact subsets of $\R^d$, and that there  is a positive constant $L$ such that $[0,T]\times B\big(\xi,LT^\alpha\big)\subset\Omega$, and 
\begin{align*}
\big\|F(t,a)\big\|\leq L,\quad \forall (t,a)\in [0,T]\times B\big(\xi,LT^\alpha\big).
\end{align*}
Then there exists a time horizon  $T^*\in (0,T)$ and a solution $z$ to the rough differential inclusion \eqref{EqRDI} stated at $\xi$, defined on the time interval $[0,T^*]$.
\end{thm}

Note that the regularity of $G$ is lower in Theorem~\ref{ThmRDI} than in Theorem~\ref{ThmRDIUSD}. When $G=0$ in Theorem~\ref{ThmRDIUSD} and Theorem~\ref{ThmRDI}, we recover classical conditions for existence of solutions to differential inclusions -- see e.g. the classical reference \cite{AubinCellina}. In Theorem~\ref{ThmRDI}, when $F=\{0\}$, we recover the optimal condition of regularity on $G$ for existence of solutions to a rough differential equation \cite{Davie}.

\medskip

The work has been organized as follows. Section \ref{SectionProof} is dedicated to proving  Theorem \ref{ThmMain} on Young differential inclusions, following the above strategy. Theorem \ref{ThmSideSelection} is proved in Section \ref{SectionSideSelection}, and Section \ref{SectionRDI} is dedicated to the proof of  Theorem \ref{ThmRDI}.

\section{Young differential inclusions}
\label{SectionProof}

This section is dedicated to proving Theorem \ref{ThmMain}. Given $q\geq 1$ and a finite dimensional vector space $E$, we denote by $V_q\big([0,T], E\big)$ the space of $E$-valued paths with finite $q$-variation on the time interval $[0,T]$. Refer to Appendix \ref{SectionAppendix}.

\ssk

We fix \textcolor{black}{$\alpha\in (\frac{1}{2},1]$} once and for all in this section,
and work in the setting of Theorem \ref{ThmMain}. The pattern of proof
of Theorem \ref{ThmMain} goes as follows. Recall from Appendix
\ref{SectionAppendix} that the space of paths with finite
$p$-variation is endowed with the $p$-variation norm $\|\cdot\|_{\textrm{$p$-var},
  \infty}$, defined in \eqref{EqDefnNormpVar}. Given a positive finite time
horizon $T$, we define the dyadic partitions $\pi^{(m)}=\{{\color{black}t}^m_i\}$ of
the interval $[0,T]$, with ${\color{black}t}^m_i := i2^{-m}T, \,0\leq i\leq
2^m$ and $m\geq 1$. Let $\gamma$ be  given in Theorem \ref{ThmMain} and $p\geq 1$ be such that $\frac{\gamma}{p}+\alpha>1$. We construct in Section \ref{SubsectionConstruction} an approximate solution to the problem on a sufficiently small time interval $[0,T]$, under the form of a pair $(z^m,v^m)$ such that 
\begin{itemize}
   \item $v^m$ has $\|\cdot\|_{\textrm{$p$-var}, \infty}$-norm uniformly bounded in $m$, and is equioscillating,   \vspace{0.1cm}
   
   \item $v^m_t\in F(z^m_t)$ for all dyadic times, and
   \begin{align*}
   z^m_t = \xi+\int_0^t v^m _udx_u, \quad 0\leq t\leq T,
   \end{align*}
with $\xi\in\R^d$.
 \end{itemize}

 It follows then from the first item that the sequence $v^m$ has a
 converging subsequence $v^{m_k}$ with limit some $v$, for the
 $\|\cdot\|_{\textrm{$q$-var}, \infty}$-norm, for any $q>p$ with
 $\frac{\gamma}{q}+\alpha>1$. The continuity statement from Corollary
 \ref{CorAppendix} implies then that $z^{m_k}$ converges in
 $\|\cdot\|_{\alpha, \infty}$-norm to the path $z := \xi+\int_0^\cdot
 v_udx_u$. One gets the fact that $v_t\in F(z_t)$, for all $0\leq
 t\leq T$, from the fact that $F$ is bounded and takes values in
 closed sets. \textcolor{black}{Thus $(z,v)$ is a solution in $C^{\alpha}\big([0,T],{\mathbb R}^d\big)\times V_q\big([0,T], L( {\mathbb R}^\ell,{\mathbb R}^d)\big).$} The existence of a solution to the inclusion defined up to the initial time horizon is a consequence of the fact that the previous existence time does not depend on $\xi$. We now turn to the details.

\subsection{Construction of the approximate solution}
\label{SubsectionConstruction}

For $t\in \underset{m\geq 0}{\bigcup} \pi^{(m)}$, set
\begin{equation*}
M(t):=\min\big\{j; t \in \pi^{(j)}\big\}
\end{equation*}
and define the ancestor $s(t)$ of $t$ as 
\begin{align}\label{0.2}
s(t):= \max\Big\{s\in \pi^{(M(t)-1)};  s<t\Big\}.
\end{align}
For each $m\geq 0$, we construct the path $z^m$ on $[0, t_{i+1}^m]$, and $v^m$ on $[0,t_{i+1}^m)$, recursively on $0\leq i\leq 2^m-1$ {\color{black}where $t_{i+1}^m=2^{-m}T(i+1)$}. The construction is \textit{not} inductive on $m$. 

$\bullet$ For $m=0$, choose $v_0^0 \in F({\color{black}\xi})$, and set
\begin{equation*}
v^0_t=v^0_0,\quad z^0_t=\xi+ v^0_0 x_{0,t}, \qquad \forall t\in [0,T].
\end{equation*}

$\bullet$ Pick $m \geq 1$.We set 
\begin{align*}
v_t^m &:= v_0^0,\quad \forall t\in [0,T2^{-m}),\\
z^m_t &:= \xi + v_{0}^0  x_{0,t},\quad \forall t\in [0,T2^{-m}].
\end{align*}
This starts the induction over $0\leq i\leq 2^m-1$. If $z^m : [0,t_i^m]\rightarrow\bbR^d$ and $v^m : [0,t_i^m)\rightarrow\bbR^d$, have been constructed, use the H\"older continuity of $F$ to choose $v^m_{t_i^m} \in F({\color{black}z^m_{t_i^m}})$ such that 
\begin{equation*}
\Big| v^m_{t_i^m}- v^m_{s(t_i^m)}\Big| \leq \|F\|_{\gamma} \Big|z^m_{t_i^m}- z^m_{s(t_i^m)}\Big|^\gamma,
\end{equation*}
and set
\begin{align*}
v^m_t &:= v^m_{t_i^m},\quad \forall t\in [t_i^m, t_i^m + T2^{-m}),\\  z^m_t &:= z^m_{t_i^m} + v^m_{t_i^m}\,x_{t_i^m t},\quad \forall t\in [t_i^m, t_i^m + T2^{-m}].
\end{align*}
If $t_i^m + 2^{-m} T=T$, set
$v^m_T= v^m_{t_i^m}$.

We have $z^m= \xi +\int_0^t v^m_u dx_u$, as a consequence of the fact that $v^m$ is constant along the {\color{black}intervals} of the partition $\pi^{(m)}$ of $[0,T]$. The next section is dedicated to proving a uniform $p$-variation bound on a small time interval satisfied uniformly by the paths $v^m$.

\subsection{Study of the $p$-variation norm of $v^m$ in a small time interval}

This section is dedicated to proving the following intermediate result.

\begin{prop} \label{Prop} 
Pick $\frac{1}{1+\gamma}<\beta <\alpha$, and set $T_0$ as equal to 
\begin{equation}
\begin{split} 
 \label{EqDefnT0}
\min\left\{ 1, \big(2  \|F\|_{\infty} \|x\|_{\alpha}\big)^{- \frac{2}{\alpha-\beta}}, \left(   \frac{ 2\|F\|_{\gamma}\|x\|_{\alpha}}{ 1- 2^{-(\alpha + \beta \gamma -1)}}\right)^{-\frac{2}{\alpha -\beta}}, \left(\frac{\|F\|_{\gamma}}{1-2^{-\gamma\beta}}\right)^{\frac{-4}{\gamma (\alpha -\beta)}}\right\}.
\end{split}
\end{equation}
Pick $p > 1/(\gamma \beta)$. Then we have, for any $S\leq T_0$,
\begin{equation}
\label{EqMainEstimate}
\|v^m \|_{p-\textrm{\emph{var}},[0,S]} \leq \left( \frac{ 1}{ 1- 2^{-\gamma \beta p}}\right)^{1/p} S^{\frac{(\alpha- \beta)\gamma}{4} + \gamma \beta}.
\end{equation}
\end{prop}

The proof of Proposition \ref{Prop} proceeds in several steps. We first give a discrete integral representation of $z^m$ that proves useful.

\begin{lem}
\label{LemRepresentation}
Pick $0 \leq n \leq m$, and two consecutive points $s,t$ in $\pi^{(n)}$. Then, setting $s_i^k = s+ i2^{-n-k}T$, we have
\begin{align*}
z^m_{s,t}= v_s^m x_{s,t} + \sum_{k=0}^{m-n-1} \sum_{i=0}^{2^k -1} v_{s_i^{k},s_{2i+1}^{k+1}}^mx_{s_{2i+1}^{k+1},s_{i+1}^{k}}.
\end{align*}
\end{lem}

\begin{proof}
For $0\leq s\leq t\leq T$, set
\begin{equation*}
\mu^m_{s,t} := \mu^{m,(0)}_{s,t}= {\color{black} v^m_s} x_{s,t}
\end{equation*}
and, for $k \geq 1$,
\begin{equation*}
\mu^{m,(k)}_{s,t} :=\mu^{m,(k-1)}_{s,\frac{t+s}{2}}+\mu^{m,(k-1)}_{\frac{t+s}{2},t}.
\end{equation*}
$\bullet$ We first prove by descending induction on $n$ that for $0 \leq n \leq m$, and two consecutive points $s,t$ in $\pi^{(n)}$, we have
\begin{align}\label{EqIdentity}
z^m_{s,t} = \mu^{m,(m-n)}_{s,t} = \mu^{m,(k)}_{s,t}, \qquad \forall k \geq m-n.
\end{align}
This identity holds true when $n=m$ as a consequence of the definition of the objects. Assume that \eqref{EqIdentity} holds true for $n\leq m$, and let $s,t$ be two consecutive points in $ \pi^{(n-1)}$. Since \textcolor{black}{$(s,\frac{s+t}{2})$ and $(\frac{s+t}{2},t)$ are two consecutive pairs }of points in $\pi^{(n)}$, we have, from the definition of $\mu^{m,(k)}$ and induction,
\begin{align*}
z^{m}_{s,t} = z^m_{s,\frac{s+t}{2}}+ z^m_{\frac{s+t}{2},t} =\mu^{m,(m-n+1)}_{s,t}.
\end{align*}
For $k \geq {\color{black}(m-n)}$, we have
\begin{align*}
\mu^{m,(k)}_{s,t}=\mu^{m,(k-1)}_{s,\frac{s+t}{2}}+ \mu^{m,(k-1)}_{\frac{s+t}{2},t} =\mu^{m,(m-n)}_{s,\frac{s+t}{2}}+ \mu^{m,(m-n)}_{\frac{s+t}{2},t} =\mu^{m,(m-n+1)}_{s,t};
\end{align*}
this closes the inductive proof of identity \ref{EqIdentity}.

\smallskip

$\bullet$ One then sees by induction on $k$ that setting as in the statement of the Lemma
\begin{align*}
s_i^k = s+ i 2^{-k}(t-s) =  s+ i2^{-n-k}T,
\end{align*}
one has
\begin{align*}
\mu^{m,(k)}_{s,t} = \sum_{i=0}^{2^k-1}\mu_{s_i^k, s_{i+1}^k}^m
\end{align*}
and
\begin{align}\label{13-1}
\mu^{m,(k+1)}_{s,t} - \mu^{m,(k)}_{s,t}=\sum_{i=0}^{2^k -1} v_{s_i^{k},s_{2i+1}^{k+1}}^m \, x_{s_{2i+1}^{k+1},s_{i+1}^{k}}.
\end{align}
Summing equation \eqref{13-1} for $k$ from $0$ to $m-n$, and using identity \eqref{EqIdentity} gives the identity of the Lemma. 
\end{proof}

\begin{cor}
\label{CorUnifEstimateZ}
Pick $\frac{1}{1+\gamma}<\beta<\alpha$, and set
\begin{align*}
T_1 := \min\left\{ 1, \left(2  \|F\|_{\infty} \|x\|_{\alpha}\right)^{- \frac{2}{\alpha-\beta}}, \left(\frac{ 2\|F\|_{\gamma}\|x\|_{\alpha}}{ 1- 2^{-(\alpha + \beta \gamma -1)}}\right)^{-\frac{2}{\alpha -\beta}}\right\}.
\end{align*}
Then we have for any $0\leq S\leq T_1$ and any $0\leq n\leq m$, the $m$-uniform bound  
\begin{equation}
\label{EqUniformEstimate}
\sup_{[s,t] \in \pi^{(n)}} |z_{s,t}^m| \leq S^{\frac{\alpha-\beta}{2}} \, (S2^{-n})^{\beta};
\end{equation}
the supremum is over consecutive points $s,t$ of $\pi^{(n)}$.
\end{cor}

\begin{proof}
The proof is again by descending induction on $n\in\{ 0,\dots,m\}$. We first have for two consecutive points $s,t$ of $\pi^{(m)}$ the estimate 
\begin{align*}
|z^m_{s,t}| \leq \|F\|_{\infty} \|x\|_{\alpha} (S2^{-m})^{\alpha},
\end{align*}
so \eqref{EqUniformEstimate} holds true for $m=n$ since $0\leq S\leq T_1$. Assume now that \eqref{EqUniformEstimate} has been proved for $n\leq m$, an let $s,t$ be two consecutive points of $\pi^{(n-1)}$. We use the representation formula 
\begin{align*}
z^m_{s,t}= v_s^m  x_{s,t} + \sum_{k=0}^{m-n+1} \sum_{i=0}^{2^k -1}
v^m_{s_i^{k},s_{2i+1}^{k+1}}x_{s_{2i+1}^{k+1},s_{i+1}^{k}}
\end{align*}
from Lemma \ref{LemRepresentation}, with $s_i^k =s + i2^{-n-k+1}S$. Note that $M(s_{2i+1}^{k+1})=n+k$, here, and the ancestor $s(s_{2i+1}^{k+1})$ of $s_{2i+1}^{k+1}$ is $s_i^k$. Then, using \eqref{EqUniformEstimate} for $n+k$, we have
\begin{align*}
\Big|v_{s_i^{k},s_{2i+1}^{k+1}}^m\Big| \leq \|F\|_{\gamma}\Big|z^m _{s_i^{k},s_{2i+1}^{k+1}}\Big|^{\gamma}\leq \|F\|_{\gamma} S^{\frac{(\alpha-\beta)}{2} \gamma} \big(S2^{-n-k}\big)^{\beta\gamma},
\end{align*}
so we obtain for $| z^m_{st} |$ the upper bounds
\begin{align*}
&\|F\|_{\infty}\|x\|_{\alpha} \big(2^{-n+1}S\big)^{\alpha} + \sum_{k=0}^{m-n+1}  2^{k}S^{\frac{(\alpha-\beta)}{2}  \gamma}\big(S2^{-n-k}\big)^{\beta\gamma}\|x\|_{\alpha}\big(S2^{-n-k}\big)^{\alpha}   \\
&\leq \|F\|_{\infty}\|x\|_{\alpha} \big(2^{-n+1}S\big)^{\alpha} + \|x\|_{\alpha} \|F\|_{\gamma} S^{ \alpha + \gamma\beta+\frac{(\alpha-\beta)}{2}  \gamma }\,\frac{2^{-n( \alpha + \beta \gamma )}}{ 1- 2^{-(\alpha + \beta \gamma -1)}}.
\end{align*}
The choice of $S\leq T_1$ ensures that 
\begin{align*}
\|F\|_{\infty}\|x\|_{\alpha} (2^{-n+1}S)^{\alpha} \leq \frac{1}{2} S^{\frac{\alpha-\beta}{2}} ( 2^{-(n-1)} S)^{\beta}
\end{align*}
and since $\alpha + \gamma \beta >1 >\alpha$ and $S<1$
\begin{align*}
\|x\|_{\alpha} \|F\|_{\gamma} S^{ \alpha + \gamma\beta+ \frac{(\alpha-\beta)}{2}  \gamma }\,\frac{2^{-n( \alpha + \beta \gamma )}}{ 1- 2^{-(\alpha + \beta \gamma -1)}} \leq \frac{1}{2} S^{\frac{\alpha- \beta}{2}} \big( 2^{-(n-1)} S\big)^{\beta};
\end{align*}
this closes the descending induction step.
\end{proof}

Recall the oscillation ${\sf Osc}(v,I)$ of a function $v : I\mapsto\bbR^d$, is defined by the formula
\begin{align*}
{\sf Osc}(v,I) := \sup_{a,b\in I}\big(v(b)-v(a)\big).
\end{align*}
The uniform control of the oscillation of the $v^m$ provided by the next statement is necessary to use the compactness result on the set of bounded functions equipped with uniform norm stated in Theorem 5, Section 4, Chapter 0 of Aubin and Cellina's book \cite{AubinCellina}. Recall the definition of $T_0\leq T_1$ from the statement of Proposition \ref{Prop}. The notation $[s,t] \in \pi^{(n)}$ used below stands for two consecutive points $s,t$ in $\pi^{(n)}$.

\begin{cor}
\label{CorOscillation}
For any $0\leq S\leq T_0$, we have, for any $0\leq n\leq m$,
\begin{equation}
\label{EqControlOscillation}
\sup_{[s,t] \in \pi^{(n)}} {\sf Osc}\big(v^m,[s,t)\big) \leq S^{\frac{(\alpha-\beta)\gamma}{4}}\big(S2^{-n}\big)^{\gamma \beta}.
\end{equation}
\end{cor}

\begin{proof}
Note that if $n\geq m$ and  $s$ and $t$ are two consecutive points of $\pi^{(n)}$, the function $v^m$ has null oscillation on the interval
$[s,t)$, since it is constant on the intervals of the partition $\pi^{(m)}$. Let then take $n\leq (m-1)$. Set $s_0=s$ and define a finite sequence $(s_i)_{i=0,...,m-n}$ setting $s_{i+1}=s_i+2^{-(n+i+1)}S$, if $s_i+2^{-(n+i+1)}S\leq t$, and $s_{i+1}=s_i$, otherwise. Then $s_i\in\pi^{(n+i)}$, for any $i$, and either $s_{i+1}=s_i$ or its ancestor $s(s_{i+1})$ is $s_i$. We then have from the uniform estimate \eqref{EqUniformEstimate} on $z^m_{s,t}$ the bound
\begin{equation*}
\big|v^m_{s_i,s_{i+1}} \big| \leq \|F\|_{\gamma} \, \big| z^m_{s_i,s_{i+1}}\big|^{\gamma} \leq \|F\|_{\gamma} \, \left(S^{\frac{{\color{black}\alpha-\beta}}{2}} (S2^{-n-i})^{\beta}\right)^{\gamma}.
\end{equation*}
We obtain \eqref{EqControlOscillation} summing these inequalities for $i$ from $0$ to $(m-n)$, and from the definition of $T_0$.
\end{proof}

\begin{proof}[Proof of Proposition \ref{Prop}]
Take $0\leq S\leq T_0$, and let $\pi = (s_i)_{i=0}^N$ be a partition of the interval $[0,S]$.

The following partition is a kind of greedy partition of the set $\{0,...,N-1\}$, in terms of the size of corresponding increments in the above formula. Let then set
\begin{align*}
 I_0 &:= \Big\{\ell\in \{0,..,{\color{black}N-1}\}; \exists j \in \{0, ...,2^m-1\} , (s_\ell, s_{\ell+1}] \subset [t_j^m, t_{j+1}^m)\Big\},\\
I_1 &:= \Big\{\ell\in \{0,..,{\color{black}N-1}\}; 2^{-1} \in (s_\ell, s_{\ell+1}]\Big\},
\end{align*}
and, for $2\leq j\leq m$, set  
\begin{align*}
I_j := \Big\{ \ell\in \{0,..,{\color{black}N-1}\}; \exists t\in \pi^{(j)}\setminus \pi^{(j-1)}\cap (s_{\ell},s_{\ell+1}]\Big\}.
\end{align*}
We define a partition of $\{ 0,\dots,N-1\}$ setting $K_0:=I_0$ and 
\begin{align*}
K_j := I_j \setminus \bigcup_{k=0}^{j-1} I_k,
\end{align*}
for $2\leq j\leq m$. Note that $K_j$ has at most $2^j$ elements.   \vspace{0.1cm}

\begin{itemize}
   \item For $\ell\in K_1$, we know from Corollary \ref{CorOscillation} on oscillations of $v^m$, with ${\color{black}n}=0$, that
\begin{align*}
\big|v^m_{s_{l+1}}-v^m_{s_l}\big| \leq S^{\frac{(\alpha- \beta)\gamma}{4}+\gamma \beta }.
\end{align*}
   
	\item \textcolor{black}{ For $\ell\in K_0$, one has $v^m_{s_j,s_{j+1}}=0$, since $v^m$ is constant over the intervall of $\pi^{(m)}.$}

   \item For $\ell\in K_j$, there exists $t \in \pi^{(j)} \setminus \pi^{(j-1)}$ such that {\color{black} $s_{\ell} <t \leq s_{\ell+1}$, and for all $u \in \pi^{(j-1)}$, one has either $u\leq s_{\ell}$, or $u >s_{\ell+1}$.}
\end{itemize}

Then  $[s_\ell,s_{\ell+1}) \subset [t-S2^{-j},t+S{\color{black}2^{-j}})$, and using Corollary \ref{CorOscillation} with $j-1$, one gets
\begin{align*}
\big|v^m_{s_{\ell+1}}-v^m_{s_\ell}\big| \leq S^{\frac{(\alpha- \beta)\gamma}{4} }\big(S2^{-(j-1)}\big)^{\gamma \beta}.
\end{align*}
Taking $p>{\color{black}1/(\gamma\beta)}$, one then has
\begin{align*}
\sum_{i=0}^{N-1} \big|v^m_{s_{i+1}} -v^m_{s_i}\big|^p \leq 2 S^{\frac{(\alpha- \beta)\gamma p}{4} }\sum_{j=0}^{m} 2^j \big(S2^{-(j-1)}\big)^{\gamma \beta p} \leq 2 S^{\frac{(\alpha- \beta)\gamma p}{4} } S^{\gamma \beta p} \frac{ 1}{ 1- 2^{-\gamma \beta p}},
\end{align*}
from which \eqref{EqMainEstimate} follows
\begin{align*}
\|v^m\|_{p-\textrm{var},[0,S]} \leq \left( \frac{ 1}{ 1- 2^{-\gamma \beta p}}\right)^{1/p} S^{\frac{(\alpha- \beta)\gamma}{4} + \gamma \beta}.
\end{align*}
\end{proof}

\subsection{Local and global existence for solutions}
\label{SectionLocalGlobal}

We finally give the proof of Theorem \ref{ThmMain} in this section. We first prove the existence of a solution to the differential inclusion \eqref{EqYDI} on the time interval $[0,T_0]$, for the time $T_0$ defined in \eqref{EqDefnT0} in Proposition \ref{Prop}. Since the definition of $T_0$ does not involve the initial condition $\xi$ of the dynamics, we obtain by concatenation a solution to the inclusion defined over the whole interval $[0,T]$.

\medskip

$\bullet$ Recall that a sequence of bounded functions $(y^m)_{m\geq 0}$ from compact segment $[a,b]$ into a compact set $K$ is said to be \emph{equioscillating} if, and only if, one can associate to any positive $\varepsilon$ a finite partition $(J_k)_{0\leq k\leq r}$ of $[a,b]$ into subintervals such that ${\sf Osc}(y^m,J_k) \leq \varepsilon$, uniformly in $k,m$. The Ascoli-Arzela-type convergence theorem from Theorem 5 of section 4 of chap 0 of Aubin and Cellina's book \cite{AubinCellina}, ensure the existence of a subsequence $(y^{m_k})_{k\geq 0}$ that converges uniformly to some limit bounded function $y$. Since the family $(v^m)_{m\geq 0}$ is bounded by $\|F\|_\infty$, and equioscillating, from Corolary \ref{CorOscillation}, it has a uniformly converging subsequence, with limit $v$.

Pick then $q >p$ such that $\frac{1}{q} +\alpha >1.$ Since all the $v^m$ have the same starting point, we have the elementary interpolation bound
{\color{black}
\begin{align*}
&\big\|v^m-v^n\big\|_{q-\textrm{var},[0,T_0]}\\
&\leq \left( 2 \|v^m-v^n\|_{\infty,[0,T_0]}\right)^{\frac{q-p}{q}}\left(\|v^m\|_{p-\textrm{var},[0,T_0]}+\|v^n\|_{p-\textrm{var},[0,T_0]}\right)^{p/q},
\end{align*}
}
on which we read off the convergence of a subsequence $z^{m_k}$ of the $v^m$ in $q$-variation norm to its limit, as a consequence of the uniform bound from Proposition \ref{Prop}. The convergence to $v$ of this subsequence is thus in the sense of the $\|\cdot\|_{q-\textrm{var},[0,T_0],\infty}$-norm.

\smallskip

$\bullet$ The continuity result on Young integrals recalled in
Corollary \ref{CorAppendix} from Appendix \ref{SectionAppendix}
implies then the convregence in the norm $\|\cdot\|_{\alpha,\infty}$
of $z^{m_k}$ on $[0,T_0]$ to the path $z$ defined by the equation
\begin{align*}
z_t = \xi+\int_0^t v_udx_u.
\end{align*}

$\bullet$ It remains to prove that $v_t\in F(z_t)$, for all times $0\leq t\leq T_0$. For a dyadic time $t \in \bigcup_{m\geq 0} \pi^{(m)}$, then for $k$ big enough, one has $v^{m_k}_t \in F( z^{m_k}_t)$. Then, since $F$ is $\gamma$-H\"older, one has 
\begin{align*}
d\big(v_t,F(z_t)\big) \leq \big|v_t-v^{m_k}_t\big| + d\big(F( z^{m_k}_t), F(z_t)\big) \leq \big|v_t-v^{m_k}_t\big| + \|F\|_{\gamma}\,\big| z^{m_k}_t-z_t\big|^{\gamma};
\end{align*}
so one gets $d(v_t,F(z_t))=0$, and $v_t \in F(z_t)$, since $F(z_t)$ is closed. For an non-dyadic time $t \in [0,T_0]$, there exists two consecutive points in $\pi^{(m)}$ such that $t \in [u,v[$, for every $m\geq 0$. Then,
\begin{equation*}
\begin{split}
d\big(v_t,F(z_t)\big) \leq \big|v_t -v^m_t\big| &+ \big|v^m_t- v^m_u\big|  + d\big(v^m_u, F(z^m_u)\big) + d\big( F(z^m_u), F(z_u)\big)   \\
&+ d\big(F(z_u),F(z_t)\big),
\end{split}
\end{equation*}
while we have from Corollary \ref{CorOscillation} and Corollary \ref{CorUnifEstimateZ}
\begin{align*}
d\big(v_t,F(z_t)\big) \leq \big|v_t -v^m_t\big| + (T_0{\color{black}2^{-m}})^{\gamma \beta} + \|F\|_{\gamma}\big|z^m_u-z_u\big|^{\gamma} + \|F\|_{\gamma}\big|z_u-z_t\big|^{\gamma}
\end{align*}
and
\begin{equation*}
\begin{split}
d\big(v_t,F(z_t)\big) \leq &\big\|v-v^{m_k}\big\|_{\infty, [0,T]} + (T_0{\color{black}2^{-m_k}})^{\gamma \beta} \\
&+ \|F\|_{\gamma} \Big(\big\|z-v^{m_k}\big\|_{\infty, [0,T]} +\|z\|_{\alpha,[0,T]}(T_0{\color{black}2^{-m_k}})^{\gamma \alpha} \Big).
\end{split}
\end{equation*}
So one gets $d(v_t,F(z_t))=0$, and $v_t \in F(z_t)$, since $F(z_t)$ is closed.

\section{A selection result}
\label{SectionSideSelection}

As a by product of the proof of Theorem \ref{ThmMain}, we obtain the selection result of Theorem~\ref{ThmSideSelection} wich partially answers Remark 8.1 of Chistyakov and Galkin's work \cite{Chistyakov}. This section is dedicated to proving Theorem~\ref{ThmSideSelection}.

\medskip

We use the same notations an in Section \ref{SectionProof}. Define for each non-negative integer $m$ the partition $\pi^m := \{t^m_i\}_{i=0..2^m}$ of the interval $[0,1]$, with $t^m_i := i2^{-m}$. We define as follows a path $x^m: [0,1]\mapsto\bbR^d$, on each sub-interval $[0,t_i^m)$, recursively over $i$.
\begin{itemize}
\item 
For $m=0$, set $x^0(t)= x_0$, for all $0\leq t\leq 1$.

\item For $m \geq 1$, set first $x^m(t)= x_0$, on $[0,t_1^m)$, and assuming $x^m$ has been constructed on the time interval $[0,\tau)$, set $x^m_T=x^m_{\tau-2^{-m}}$, if $\tau= T$, otherwise choose $x^m_{\tau} \in F(\tau)$ such that 
\begin{align*}
d\Big(x^m_{\tau}; x^m_{s(\tau)}\Big) \leq \|F\|_{\gamma} 2^{(-M(\tau)+1)\gamma}
\end{align*}
and set $x^m_t=x^m_{\tau}$, for $\tau \leq t<\tau +2^{-m}$.
\end{itemize}

We first prove that we have for all $r,m$ the estimate
\begin{align}
\label{eq-osc-1}
\max_{t \in \pi^r}\sup_{s \in [t,t+ 2^{-r})} \big|x_{s,t}^m\big| \leq \frac{ \|F\|_{\gamma}}{ 1-2^{-\gamma}}\,2^{-r\gamma}.
\end{align}
By construction, $x^m$ is constant on the each {\color{black} interval} of $\pi^m$, so if $r\geq m, t\in \pi^r$ and $s\in [t,t+ 2^{-r})$, then $x_{s,t}^m=0$ and \eqref{eq-osc-1} holds true. Let then consider the case where $r <m$. We define a finite sequence of times $(s_i)_{i=0,...,m-r}$, with $s_0=t$, such that for $0\leq i \leq (m-r-1)$, we have $s_{i+1}=s_i+ 2^{-r-i-1}$, if $s_i+ 2^{-r-i-1} \leq s$, and $s_{i+1}=s_i$ otherwise. Then, for $s_i \in \pi^{r+i}$, for $i\geq 1$, and either $s(s_{i+1})=s_i$ or $s_{i+1}=s_i$. The path $x^m$ is constructed in such a way as to have
\begin{align*}
\big|x^m_{s_i,s_{i+1}}\big| \leq \|F\|_{\gamma} 2^{-(r+i-1)\gamma}.
\end{align*}
Summing these estimates for $0\leq i\leq m-r$, gives \eqref{eq-osc-1}, and 
\begin{align}\label{eq-osc-2}
\max_{t \in \pi^r} \textsf{Osc}\big(x^m, [t,t+ 2^{-r})\big) \leq \frac{2 \|F\|_{\gamma}}{ 1-2^{-\gamma}}\,2^{-r\gamma}.
\end{align}
Since $x^m $ is constant along the sub-intervals of the partition $\pi^m$ it is enough, to compute the $p$-variation of $x^m$ along any partition $\pi=\{s_i,~~i=0,...,N\}$ of $[0,1]$, to assume that all the partition times $s_i\in\pi^{(m)}$. Let then define as follows be a finite partition $(K_j)_{j=0,...,m}$ of $\{ 0,\dots,N-1\}$, with possibly empty sets ~:
\begin{align*}
I_0 &:= \Big\{\ell\in \{0,..,{\color{black}N-1}\}; \exists j \in \{0, ...,2^m-1\} , (s_\ell, s_{\ell+1}] \subset [t_j^m, t_{j+1}^m)\Big\},\\
I_1 &:= \Big\{\ell\in \{ 0,\dots,N-1\};\;2^{-1} \in (s_\ell,s_{\ell+1}]\Big\},   \\
I_j &:= \Big\{ \ell\in \{ 0,\dots,N-1\}; \; \exists t\in \pi^j\setminus \pi^{j-1},~~t \in (s_\ell,s_{\ell+1}]\Big\},\quad 2\leq j\leq m,
\end{align*}
with $K_0:=I_0$ and 
\begin{align*}
K_{j+1} &:= I_{j+1}\setminus \left( \bigcup_{k=0}^j I_k\right).
\end{align*}
For $\ell\in K_1,$ using \eqref{eq-osc-2} for $r=0$ we bound 
\begin{align}\label{18-0-x}
\big|x^m_{s_{\ell+1}}-x^m_{s_\ell}\big| \leq \frac{2 \|F\|_{\gamma}}{ 1-2^{-\gamma}} .
\end{align}
For $\ell\in K_j,$ then there exists $t \in \pi^j \setminus \pi^{j-1}$ such that 
\begin{align*}
{\color{black}s_{\ell} <t \leq s_{\ell+1},}
\end{align*}
and any $u \in \pi^{j-1}$ satisfies either {\color{black}$u\leq s_{\ell}$ or $u >s_{\ell+1}$}. Then  
\begin{align*}
[s_\ell,s_{\ell+1}) \subset [t-2^{-j},t+{\color{black}2^{-j}}),
\end{align*} and using inequality \eqref{eq-osc-2} for $j-1$, we see that
\begin{align}\label{18-1-x}
\big|x^m_{s_{\ell+1}}-x^m_{s_\ell}\big| \leq \frac{2 \|F\|_{\gamma}}{ 1-2^{-\gamma}}\,2^{-(j-1)\gamma}.
\end{align}
The number of indices $\ell\in K_j$ is at most $2^{j}$. Using the fact that the family $(K_j)_{1\leq j\leq m}$ defines a partition of {\color{black}$\{0,\dots,N-1\}$}, and adding inequalities \eqref{18-0-x} and \eqref{18-1-x} for $q >1/ \gamma$, we then have
\begin{align*}
\sum_{i=0}^{N-1} \big|x^m_{s_{i+1}} -x^m_{s_i}\big|^q &\leq  \left(\frac{2 \|F\|_{\gamma} 2^{\gamma}}{ 1-2^{-\gamma}}\right)^q \sum_{j=0}^{m} 2^j {\color{black} 2^{-j\gamma q}}\\
&\leq   \left(\frac{2 \|F\|_{\gamma} 2^{\gamma}}{ 1-2^{-\gamma}}\right)^q \frac{ 1}{ 1- 2^{-\gamma  q+1}}
\end{align*}
and we derive 
\begin{align*}
 \|x^m \|_{q-\textrm{var}} \leq \frac{2 \|F\|_{\gamma} 2^{\gamma}}{ 1-2^{-\gamma}} \frac{ 1}{ (1- 2^{-\gamma  q+1})^{1/q}}.
\end{align*}
From Helly' selection principle, see e.g. Theorem 6.1 in Chistyakov and Galkin's work \cite{Chistyakov}, the sequence $x^m$ has a convergent subsequence in $p$-variation, for any $p>q$. We identify the path $f$ from the statement as such a limit. One proves that $f$ is a selection of $F$ in the same way as we proved that $v_t\in F(z_t)$ in Section \ref{SectionLocalGlobal}, using the fact that $F(t)$ is closed for all times, and the regularity properties of $F$.

\section{Rough differential inclusions}
\label{SectionRDI}

This section is dedicated to the proof of Theorem \ref{ThmRDIUSD} and Theorem \ref{ThmRDI}. We refer the reader to Appendix \ref{SectionAppendix} for basics on rough paths theory. All we need to know is recalled there. The reader will notice that the proof of Theorem \ref{ThmRDIUSD} is much shorter than the proof of Theorem \ref{ThmRDI}. This is due to the fact that we somehow assume much more in Theorem \ref{ThmRDIUSD}, asking in particular that the set-valued drift has compact \textit{convex} images. Together with the assumed upper semicontinuity, this allows for the use of powerful approximate selection theorems that greatly simplify the matter. See e.g. Section 1 of Chapter 2 in \cite{AubinCellina} for the differential inclusion case.

\subsection{Upper semicontinuous drift}
\label{SubsectionUSC}

This section is dedicated to proving Theorem \ref{ThmRDIUSD}; we work in this section in the setting of that statement. Here, $\rX$ is an $\alpha$-H\"older weak geometric rough path with $\alpha\in (1/3,1/2]$ and $G\in {\color{black}C^{2,\gamma}_b}$ with $\alpha(2+\gamma)>1$. 
We recall that according to the notation introduced in the Introduction, $C^1([0,T],\R^d)$ is the space of $\R^d$-valued Lipschitz paths
defined on $[0,T]$. Given $x\in C^1([0,T],\R^d)$, we denote by $\psi(t,x)$ the solution path to the rough differential equation
\begin{equation} \label{EqUSCDrift}
y_t = \xi + x_t + \int_0^t G(y_s)d{\bf X}_s,
\end{equation}
with fixed initial condition $\xi$. The controlled solution path comes under the form of a pair $\big(y_t,G(y_t)\big)$. Classical results from rough paths theory ensure that $\psi$ is a continuous function from $C^{1}([0,T],\R^d)$ with values in the space $C^\alpha([0,T],\mathbb{R}^d)$ --  \textcolor{black}{see e.g. Theorem 3 of \cite{Coutin-Lejay}}. This is where we need the assumption that $G$ is $(2+\gamma)$-H\"older, rather than just $(1+\gamma)$-H\"older, as in the proof of Theorem \ref{ThmRDI} given in the next section. For each $0\leq t\leq T$, one has
\begin{align*}
\psi(t,x) = \psi(t,\overline{x}),
\end{align*}
for any other $\overline{x}\in C^1([0,T],\R^d)$ that coincides with $x$ on
the time interval $[0,t]$. We can then work in the setting of
differential inclusions with memory -- see e.g. Section 7 in Chapter 4
of Aubin \& Cellina's book \cite{AubinCellina}. Let us define 
\begin{align*}
\mathcal{F}(t,x) := F\big(t,\psi(t,x)\big).
\end{align*}
The function $\mathcal{F}$ is upper semicontinuous on
$\mathbb{R}_+\times C^1([0,T],\R^d)$, with values in the set of convex compact subsets of $\mathbb{R}^d$. We can then use the obvious variant of Theorem 1 in Section 7 in Chapter 4 in \cite{AubinCellina}, with no constraint and $K(t)=\mathbb{R}^d$ for all $t$ with the notations therein, to get the existence of a time $T^*\in (0,T]$, and a Lipschitz path $x$, defined on the time interval $[0,T^*]$, such that one has for almost all $t\in [0,T^*]$
\begin{align*}
\dot{x}_t \in \mathcal{F}\big(t,x\big).
\end{align*}
This condition is equivalent to saying that the path $y$ from \eqref{EqUSCDrift} solves the rough differential inclusion
\begin{align*}
dy_t \in F(t,y_t) + G(y_t)d{\bf X}_t.
\end{align*}
(Note that the convexity assumption on the pointwise images of $F$ is essential for the use Theorem 1 in Section 7 in Chapter 4 of Aubin-Cellina's book \cite{AubinCellina}.)

\subsection{Lower semicontinuous drift}
\label{SubsectionUSC}

This section is dedicated to proving Theorem \ref{ThmRDI}. We take advantage in that task of the approach to Filipov's theorem given by Bressan in \cite{bressan88}, based on a selection theorem of independent interest. We recall the setting before embarking on the proof of Theorem \ref{ThmRDI}.

\medskip

Recall also that a cone of $\R^m$ is a subset $\Gamma$ of $\R^m$ such that 
\begin{align*}
\Gamma\cap (-\Gamma)=\{0\},
\end{align*}
and $\lambda a\in\Gamma$, if $\lambda\geq 0$ and  $a\in\Gamma$. As an example, given a positive constant $M$, the set
\begin{align*}
\Gamma_M := \Big\{(t,x)\in \R\times\R^d\,;\,t\geq 0,\;\norm{x}\leq tM\Big\}\subset \R^{d+1}
\end{align*}
is a cone of $\R^{d+1}$.

\begin{defn*}
Let $\Gamma$ be a cone of $\mathbb{R}^m$. A map $h:\R^{m}\rightarrow\R^{d}$ is said to be \emph{$\Gamma$-continuous at point $a\in\R^m$}, if for any $\epsilon>0$, there is $\delta>0$ such that
\begin{align*}
\big\|h(b)-h(a)\big\| < \epsilon,
\end{align*}
for any $b\in B(a,\delta)\cap (a+\Gamma)$. 
\end{defn*}

We say that $h$ is $\Gamma$-continuous on a subset $S\subset\mathbb{R}^m$, if $h$ is $\Gamma$-continuous at any point of $S$. The relevance of the notion of $\Gamma$-continuity in the setting of differential inclusions is a consequence of the following selection result, due to Bressan, Theorem 1 in \cite{bressan88}.

\begin{thm}   \label{thm:bressan}
Let $H:\R^m\rightarrow \dot 2^{\R^d}$ be lower semicontinuous
set-valued map with non-empty closed values. Then, for any cone $\Gamma\subset\R^m$, there is a $\Gamma$-continuous selection of $H$.
\end{thm}

For a positive finite constant $L$, and $(\beta,\xi)\in (0,1]\times\R^d$, set
\begin{align*}
\mathcal{E}^{\beta}_L := \Big\{ u\in\CO\,;\,\|u\|_\beta\leq L,~~u_0=\xi \Big\},
\end{align*}
where $\|u\|_\beta$ is the $\beta$-H\"older norm of $u$, for $0<\beta<1$, and $\|u\|_1$ is the Lipschitz norm of $u$. The set $\mathcal{E}^{\beta}_L$ is a compact subset of $\big(\CO, \|\cdot\|_\infty\big)$. We do not emphasize the dependence on $\xi$ and $T$ in the notation.

\begin{lem}   \label{lem:pertu-select}
Assume the set-valued map $F$ is defined on a {\color{black}closed subset} $\Omega$ of $\R\times\R^d$ where it is lower semicontinuous and takes values in the set $\mathcal{K}(\R^d)$ of nonempty compact subsets of $\R^d$. Assume further that there is a positive constant $L$ such that $[0,T]\times B\big(\xi,LT^\beta\big)\subset\Omega$, and 
\begin{align*}
\big\|F(t,x)\big\|\leq L,
\end{align*}
for all $(t,x)\in [0,T]\times B\big(\xi,LT^\beta\big)$. Then there exists a continuous map 
\begin{align*}
\phi : \mathcal{E}_L^\beta \rightarrow \mathcal{E}_L^1
\end{align*}
such that one has 
\begin{equation}   \label{EqInclusionLemmaRDI}
\frac{d\phi(u)(t)}{dt}\in F(t,u(t)),\quad\text{for almost all } t\in [0,T],
\end{equation}
for any $u\in \mathcal{E}_L^\beta$.
\end{lem}

\begin{proof}
We first extend $F$ into a lower semicontinuous map $F^*$ defined on $\R\times\R^d$, setting 
\begin{align*}
F^*(t,x)=B(0,L),
\end{align*}
if $(t,x)\notin \Omega$. Pick $M>L$, and use Theorem~\ref{thm:bressan} to pick a $\Gamma_M$-continuous selection $f$ of $F^*$. For $u\in\mathcal{E}^\beta_L$, and $0\leq t\leq T$, set
\begin{align*}
\phi(u)(t) := \xi+\int_0^t f(s,u_s)ds.
\end{align*}
We have $\norm{f(t,u_t)}\leq L$, since $u\in\mathcal{E}^\beta_L$, and given the assumptions on $F$. So $\phi(u)\in\mathcal{E}^1_L$, and one has indeed the inclusion \eqref{EqInclusionLemmaRDI} for almost all times.

One can proceed as follows to show the continuity of the map $\phi$. Since $f$ is $\Gamma_M$-continuous, one can associate to any time $t\in [0,T]$ a positive $\eta_t$ such that for all $s\in [t,t+\eta_t]\textcolor{black}{\cap[0,T]}$ and $\norm{a-u_t}\leq  M\eta_t$,  one has
\begin{align}   \label{eq:f-eps}
\big\| f(s,a) - f(t,u_t) \big\| \leq \frac{\epsilon}{4T}.
\end{align}
Write $\textsc{Leb}$ for Lebesgue measure on the real line. From Lemma 1 in Bressan's work \cite{bressan88}, there is a finite family $\left(\left[t_i,t_i+\eta_{t_i}^{\frac{1}{\beta}}\right)\right)_{1\leq i\leq N}$ of disjoint intervals, such that
\begin{align}   \label{eq:measI}
\textsc{Leb}\left( \bigcup_{i=1}^N \left[t_i,t_i+\eta_{t_i}^{\frac{1}{\beta}}\right)\right)\geq T-\frac{\epsilon}{8L}.
\end{align}
Define 
\begin{align*}
\delta := \min\left(\left(\frac{\epsilon}{8NL}\right)^\alpha,\min_{1\leq i\leq N}\frac{\eta_{t_i}}{2}\right)(M-L),
\end{align*}
and set 
\begin{align*}
J_i := \left[t_i,t_i+\left(\frac{\delta}{M-L}\right)^{\frac{1}{\alpha}}\right).
\end{align*}
Then
\begin{align}   \label{eq:JL}
\textsc{Leb}\left(\bigcup_{i=1}^N J_i\right) \leq  N\left(\frac{\delta}{M-L}\right)^{\frac{1}{\alpha}}\leq \frac{\epsilon}{8L}.
\end{align}
Pick now $v\in\mathcal{E}^\beta_L$ such that $\|u-v\|_{\infty} \leq \delta$. For
\begin{align*}
t\in \left( \bigcup_{i=1}^N \big[t_i,t_i+\eta_{t_i}^{\frac{1}{\beta}}\big)\right) \setminus \left(\bigcup_{i=1}^N J_i\right) =: I\backslash J,
\end{align*}
one has $t\in \Big[t_i+\left(\frac{\delta}{M-L}\right)^{\frac{1}{\alpha}},t_i+\eta_{t_i}^{\frac{1}{\alpha}}\Big)$, for an index $i$, and
\begin{align*}
\big\| v(t) - u(t_i)\big\| &\leq \big\|v(t)-u(t)\big\| + \big\|u(t)-u(t_i)\big\|   \\ 
&\leq \delta+L(t-t_i)^\beta   \\
&\leq M(t-t_i)^\beta
\end{align*}
while
\begin{align*}
\big\|u_t-u_{t_i}\big\| \leq L(t-t_i)^\beta < M(t-t_i)^\beta.
\end{align*}
Thus, with \eqref{eq:f-eps} we obtain that, for $t\in I\backslash J$, we have
\begin{align*}
\big\| f(t,v_t)-f(t,u_t) \big\| &\leq \big\|f(t,v_t)-f(t_i,u_{t_i})\big\| + \big\|f(t_i,u_{t_i})-f(t,u_t)\big\|   \\
&\leq \frac{\epsilon}{2T}.
\end{align*}
Since $f$ is bounded by $L$ on $[0,T]\times B\big(z_0,LT^\beta\big)$, it follows that we have for any $t\in [0,T]$ the estimate
\begin{align*}
&\big\|\phi(v)(t)-\phi(u)(t)\big\| \\
& \leq \int_0^T \big\| f(t,v_s)-f(t,u_s) \big\|\,ds   \\
      &\leq \int_{I\setminus J} \big\|f(s,v_s)-f(s,u_s)\big\|\,ds + \int_{J\bigcup [0,T]\setminus I} \big\|f(s,v_s)-f(t,u_s)\big\|\,ds   \\
      &\leq T\frac{\epsilon}{2T} + \Big(\textsc{Leb}(J) + \textsc{Leb}([0,T]\setminus I)\Big)\,2L   \\
      &\leq \frac{\epsilon}{2} + \left(\frac{\epsilon}{8L}+\frac{\epsilon}{8L}\right)2L\leq \epsilon.
\end{align*}
\end{proof}

We use from now on the notations on controlled paths recalled in Appendix \ref{SectionAppendix}. Recall the one form $G$ in Theorem \ref{ThmRDI} is assumed to be {\color{black}$C^{1,\gamma}_b$, with  $1<\alpha(2+\gamma)<3\alpha$}. {\color{black} Choose a constant $\beta\in (1/3,\alpha)$ such that $\beta(2+\gamma)>	1$.}

We define a ball in the space of paths controlled by $X\in C^\alpha\big([0,T],\R^\ell\big)$, using the $\beta$-norm
\begin{align*}
B^\beta_L := \Big\{(y,y')\in \mathcal{D}^{\beta,2\beta}([0,T],\R^d)\,;\,\norm{(y,y')}_{\beta,2\beta}\leq L,~y_0=\xi, \,y'_0=G(\xi)\Big\}.
\end{align*}
It follows from Ascoli-Arzela theorem that this ball is a compact convex subset of 
\begin{align*}
C^0\big([0,T],\R^d\big)\times C^0\big([0,T],L(\R^\ell,\R^d)\big),
\end{align*}
with both factor endowed with the uniform norm. For $(y,y')\in \mathcal{D}^{\beta,2\beta}([0,T],\R^d)$, 
and $t\in [0,T]$, we set
\begin{align*}
\Phi(y,y')(t) := \left(\phi(y)(t)+\int_0^t G(y_s)d{\rX}_s,\, G(y_t)\right).
\end{align*}
We prove below that $\Phi$ is a continuous map from $B^\beta_L$ into itself, provided $T$ is small enough. See point \emph{(a)} for the \textit{stability} and point \emph{(b)} for the \textit{continuity}. It follows then from Schauder fixed point theorem that $\Phi$ has a fixed point $(\tilde{z},\tilde{z}')$ in $B^\beta_L$. One gets the fact that $(\tilde{z},\tilde{z}')\in \mathcal{D}^{\alpha,2\alpha}([0,T],\R^d)$ from the equation that it satisfies, as \eqref{eq:int_rough} yields the estimate
{\color{black}
\begin{align*}
\norm{\tilde{z}}_\alpha \leq &LT^{1-\alpha}+\normsup{G(\tilde{z})}\norm{X}_{\alpha}+\normsup{G(\tilde{z})'}\norm{\XX}_{2\alpha}\\
 &+ C\normhb{\rX}\norm{(G(\tilde{z}),G(\tilde{z})')}_{\beta,(1+\gamma)\beta}T^{(2+\gamma)\beta-\alpha},
\end{align*}
}
from which we get $\norm{\tilde{z}'}_{\alpha}\leq \norm{G}_{C^1_b}\norm{z}_\alpha$, and finally for $R^{\tilde{z}}_{st}:=\tilde{z}_{st}-\tilde{z}'_sX_{st}$

\begin{align*}
\norm{R^{\tilde{z}}}_{2\alpha}\leq &LT^{1-2\alpha}+\normsup{G(\tilde{z})'}\normhh{\XX}\\
&+C\normhb{\rX}\norm{(G(\tilde{z}),G(\tilde{z})')}_{\beta,(1+\gamma)\beta}T^{(2+\gamma)\beta-2\alpha}.
\end{align*}

The triple $\big((\tilde{z},\tilde{z}'),\phi(\tilde{z})\big)$ is then a solution of the rough differential inclusion \eqref{EqRDI}.

\medskip

\emph{(a) Stability -- } Write $R^y_{st}$ for $y_{st}-y'_sX_{st}$, for any $0\leq s\leq t\leq T$, and similary for $z$ where \textcolor{black}{$(z,z')=\Phi(y,y')$.} On the one hand, for $ (y,y')\in B^\beta_L$, one has
\begin{align*}
  \norm{y}_\beta &\leq \normsup{y'}\norm{X}_\beta +\norm{R^y}_{2\beta}T^\beta   \\
                			&\leq (\norm{y'_0}+\norm{y'}_\beta T^\beta)\norm{X}_\beta + \norm{R^y}_{2\beta}T^\beta   \\
                  			&\leq \big(\norm{G(\xi)} + LT^\beta\big)T^{\alpha-\beta}\norm{X}_\alpha + LT^\beta.
\end{align*}
We then have $y\in \mathcal{E}^\beta_L$. On another hand, according to Proposition~\ref{lem:comp:DX}, one has 
\begin{align*}
\big(G(y),G(y)'\big)\in\mathcal{D}^{\beta,{\color{black}\beta(\gamma+1)}},
\end{align*}
because $\beta\gamma < 2\beta$. Using \eqref{eq:int_rough} and the fact that $\normhhb{\XX}\leq T^{\alpha-\beta}\normhh{\XX}$, one has for any $s,t\in [0,T]$,
\begin{align*}   
 \norm{R^z_{st}} &= \norm{z_{st}-z'_sX_{st}}   \\   
  &\leq\norm{\phi(y)_{st}}+\norm{\int_s^t G(y_r)d{\rX}_r - G(y_s)X_{st}}   \\   
 &\leq  L\abs{t-s}+\normsup{G(y)'}\normhhb{\XX}\abs{t-s}^{2\beta}\\
&~~+C\normhb{\rX}\norm{(G(y),G(y)')}_{\beta,{\color{black}\beta(\gamma+1)}}\abs{t-s}^{{\color{black}\beta(2+\gamma)}}   \\
&\leq  \Big(LT^{1-2\beta} + \normsup{G(y)'}\normhh{\XX}T^{\alpha-\beta}\\
&~~+C\normhb{\rX}\norm{(G(y),G(y)')}_{\beta,{\color{black}\beta(\gamma+1)}}T^{{\color{black}\beta\gamma}}\Big)\,\abs{t-s}^{2\beta}.
\end{align*}
Using the estimate of the size of the nonlinear image of a controlled path given in \eqref{eq:comp:DX}, one has eventually for $ \normhhb{R^z}$ the upper bound
\begin{align}
\nonumber &\Big(LT^{1-2\beta}+\normCD{G}\normsup{y}\normhb{\XX}T^{\alpha-\beta}+C\normh{\rX}\norm{G}_{C^{1,\gamma}_b}(1+\norm{G(\xi)}+L)^2T^{{\color{black}\beta\gamma}}\Big)\,  \\
        \label{eq:z-z'}
  &\leq  T^{\alpha-\beta} \Big(L+\normCD{G}\normsup{y}\normhb{\XX}+C\normh{\rX}\norm{G}_{C^{1,\gamma}_b}(1+\norm{G(\xi)}+L)^2\Big)\,
\end{align}
where we use the inequalities $T^{\alpha-\beta}\geq T^{1-2\beta},T^{{\color{black}\beta\gamma}}$ in the last line. Since 
\begin{align*}
\normhhb{R^y}\leq L,
\end{align*} we also have
\begin{align}
\label{eq:supy}
\normsup{y}\leq \norm{\xi}+\norm{G(\xi)}\normhb{X}T^\beta + LT^{2\beta}.
\end{align}
Hence, we obtain from \eqref{eq:z-z'} and \eqref{eq:supy} the upper bound
\begin{equation}
\begin{split}
&T^{\alpha-\beta} \Big(L+\normCD{\sigma}\normhb{\XX}(\norm{\xi}+\norm{G(\xi)}\normhb{X}T^\beta+LT^{2\beta})\\
&+C\normh{\rX}\norm{G}_{C^{\gamma-1}_b}(1+\norm{G(\xi)}+L)^2\Big),   \label{eq:R_z}
\end{split}
\end{equation}
for $\normhhb{R^z}$. We control the $\beta$-H\"older norm of $z'$ as follows
\begin{align}
\nonumber
\normhb{z'}&\leq \normCD{G}\normhb{y}   \\
\nonumber
&\leq  \normCD{G}\Big(\normsup{y'}\normhb{X}+\normhhb{R^y}T^\beta\Big)   \\
\nonumber
&\leq \normCD{G}\Big(\big(\norm{y_0'}+\normhb{y'}T^\beta\big)\,\normh{X}T^{\alpha-\beta}+\normhhb{R^y}T^\beta\Big)   \\
\nonumber
&\leq \normCD{G}\Big(\big(\norm{G(\xi)}+LT^\beta\big)\,\normh{X}T^{\alpha-\beta}+LT^\beta\Big)   \\
&\leq \normCD{G}\Big(\big(\norm{G(\xi)}+LT^\beta\big)\,\normh{X}+L\Big)T^{\alpha-\beta}
\label{eq:z'}
\end{align}
With $T$ small enough, we obtain from \eqref{eq:R_z} and \eqref{eq:z'} that $\normDX{(z,z')}\leq L$. This shows indeed that $\Phi$ sends $B^\beta_L$ into itself.

\medskip

\emph{(b) Continuity -- } The continuity of $\Phi$ on $B^\beta_L$, for the uniform topology on $B^\beta_L$, is a direct consequence of the continuity of $\phi$ and the continuity estimate \eqref{eq:int_rough} on rough integrals.

\appendix
\section{Basics on Young and rough integrals}
\label{SectionAppendix}

Let $p\geq 1$ be given. Recall that an $\bbR^d$-valued path $x$
defined on the time interval $[a,b]$. We say that $x$ has finite
$p$-variation over that interval $[s,t]\subset [a,b]$ if 
\begin{align*}
\|y\|_{\textrm{$p$-var}, [s,t]}^p := \sup \,\sum_i | y_{s_{i+1}}-y_{s_i}|^p < \infty,
\end{align*}
where the supremum runs over the set of finite partitions $\{s_i\}$ of
the interval $[s,t]$. We denote by $V_p([a,b],\bbR^d)$ the set of
$\bbR^d$-valued path with finite $p$-variation over $[a,b]$ and by
$C^{1/p}([a,b],\R^d)$ the set of $\R^d$-valued path $\frac{1}{p}$-Hölder
continuous on $[a,b]$. Note that an element of $V_p([a,b],\bbR^d)$ need not be continuous, while $C^{1/p}([a,b]\bbR^d)\subset V_p([a,b],\bbR^d)$. Also, a continuous path with finite $p$-variation has a reparametrized version that is $\frac{1}{p}$-H\"older over its interval of definition. Refer to \cite{Chistyakov} for a reference. We endow the space of paths with finite $p$-variation with the norm
\begin{equation}
\label{EqDefnNormpVar}
\|y\|_{\textrm{$p$-var}, \infty,[a,b]} := \|y\|_{\textrm{$p$-var},[a,b]} + \|y\|_{\infty,[a,b]},
\end{equation}
where $\norm{\cdot}_{\infty,[a,b]}$ denotes the uniform norm over $[a,b]$.
Similarly, for $\alpha\in [0,1]$, we define a norm on the space of $\bbR^d$-valued $\alpha$-H\"older functions defined on the interval $[a,b]$, setting
\begin{align*}
\|x\|_{\alpha, \infty,[a,b]} := \|x\|_{\alpha, [a,b]} + \|x\|_{\infty},
\end{align*}
where $\norm{\cdot}_{\alpha,[a,b]}$ is the classical semi-norm of
$\alpha$-Hölder paths.
For these norms and semi-norms, we omit to precise the interval
$[a,b]$ when the norms and semi-norms are taken over the whole interval of the path is defined.

Given $a\leq b$, set $\Delta_{a,b}:=\{(s,t);a\leq s\leq t\leq b\}$. A \emph{control} over the interval $[a,b]$ is a map $\omega : \Delta_{a,b}\mapsto\bbR^+$ that is null on the diagonal, is non-increasing, resp. non-decreasing, as a function of each of its first, resp. second, argument, and is sub-additive
\begin{align*}
\omega(s,u)+\omega(u,t)\leq \omega(s,t),\quad\forall a\leq s\leq u\leq t\leq b.
\end{align*}
Denote by $\omega(s,t^-)$ the left limit in $t$ of the non-decreasing function $\omega(s,\cdot)$. A control is said to be \emph{regular} if it is continuous in a neighbourhood of the diagonal. As an example, for any path $y\in V_p([a,b],\bbR^d)$, the function $\omega_y(s,t):=\|y\|_{\textrm{$p$-var}, [s,t]}^{1/p}$ is a control. For an $\bbR^d$-valued $\alpha$-H\"older path $x$, with $0<\alpha<1$, the function $\omega_x(s,t):=\|x\|^{1/\alpha}_{\alpha,[s,t]}$ is a regular control.

\smallskip

From the present day perspective, Gubinelli' sewing lemma \cite{Gubinelli} offers an easy road to constructing the Young integral -- see also \cite{FdlP}. We give here a variation due to Friz and Zhang \cite{FrizZhang17}, tailor-made to our needs. Given a map $\mu : \Delta_{a,b}\mapsto\bbR^d$ and a finite partition $\pi=\{s_i\}$ of the interval $[a,b]$, set 
\begin{align*}
\mu_{\pi} := \sum_i \mu_{s_{i+1}s_i}.
\end{align*}

\begin{prop}
Let a map $\mu : \Delta_{a,b}\mapsto\bbR^d$ be given. If there exists positive exponents $\alpha_1,\alpha_2$, with $\alpha_1+\alpha_2>1$ and controls $\omega_1,\omega_2$, with
\begin{align*}
\big|\mu_{ts}-(\mu_{tu}+\mu_{us})\big| \leq \omega_1(s,u)^{\alpha_1}\,\omega_2(u,t)^{\alpha_2},   \quad \forall a\leq s\leq u\leq t\leq b,
\end{align*}
with $\omega_2$ regular, then, for any interval $[s,t]\subset [a,b]$, the $\mu_{\pi_{st}}$ converge to some limit $I_s^t(\mu)$ as the mesh of the partition $\pi_{ts}$ of $[s,t]$ tends to $0$, and one has
\begin{align*}
\big| I_s^t(\mu) - \mu_{\pi_{st}}\big| \leq C \omega_1(s,t^-)^{\alpha_1}\,\omega_2(s,t)^{\alpha_2},
\end{align*}
for some positive constant $C$ depending only on $\alpha_1,\alpha_2$.

\end{prop}

Given $x\in C^\alpha([a,b],\bbR^\ell)$ and $y\in V_p\big([a,b],L(\bbR^\ell,\bbR^d)\big)$, set
\begin{align*}
\mu_{st} := y_s(x_t-x_s),
\end{align*}
and note that 
\begin{align*}
\mu_{ts}-(\mu_{tu}+\mu_{us}) = (y_s-y_u)(x_t-x_u), \quad s\leq u\leq t,
\end{align*}
so one has for any interval $[s,t]$ the estimate
\begin{align*}
\big|\mu_{ts}-(\mu_{tu}+\mu_{us})\big| \leq \|y\|_{\textrm{$p$-var}, [s,t]}^{1/p} \, \|x\|_{\alpha,[s,t]}^\alpha.
\end{align*}

\begin{cor}
\label{CorAppendix}
If $\frac{1}{q}+\alpha>1$, the Riemann sums $\mu_{\pi_{st}}$ associated with the preceding two-index map $\mu$ converge to some limit which we denote by $\int_s^t y_udx_u$. One has
\begin{align*}
\left\|\int_0^\cdot y_udx_u\right\|_{\alpha,[s,t]} \lesssim \|y\|_{q-\textrm{var},[s,t]}\,\|x\|_{\alpha,[s,t]},
\end{align*}
and for any $\epsilon>0$ with $\frac{1-\epsilon}{q}+\alpha>1$, and $y,y'\in V_q\big([a,b],L(\bbR^\ell,\bbR^d)\big)$, one has
\begin{align*}
\left\|\int_0^\cdot y_udx_u - \int_0^\cdot y'_udx_u\right\|_{\alpha,[s,t]} &\lesssim \|y-y'\|_{\infty, [s,t]}^\epsilon\\
&\times\Big(\|y\|^{1-\epsilon}_{q-\textrm{var}, \infty, [s,t]}+ \|y'\|^{1-\epsilon}_{q-\textrm{var}, \infty, [s,t]}\Big)\,\|x\|_{\alpha,[s,t]}.
\end{align*}
\end{cor}

We refer the reader to the lecture notes \cite{LyonsStFlour, frizhairer, BailleulLN} for pedagogical accounts of rough paths theory. The following definitions and propositions will be sufficient for our needs here.

\begin{defn*}   
Let $\alpha\in (1/3,1/2)$. We say that $\rX:=(X,\XX)$ is an $\alpha$-Hölder rough path if
\begin{itemize}
    \item $X$ is an $\R^\ell$-valued $\alpha$-H\"older path on a time interval $[0,T]$,
    
    \item $\XX$ is a path from $[0,T]^2$ to $\R^\ell\otimes\R^\ell$ such
      that
      \begin{align*}
      \norm{\XX}_{2\alpha} := \sup_{s,t\in [0,T], s\ne t}\frac{\norm{\XX_{st}}}{|t-s|^{2\alpha}}<+\infty,
      \end{align*}
    \item we have
       \begin{align*}
       \XX_{rt}-\XX_{rs}-\XX_{st}=X_{rs}\otimes X_{st},
       \end{align*}
       for all $0\leq r\leq s\leq t\leq T$.
      \end{itemize}
\end{defn*}

The formula 
\begin{align*}
\norm{\rX}_\alpha:=\norm{X}_\alpha+\norm{\XX}_{2\alpha}
\end{align*}
defines a seminorm on the space of $\alpha$-H\"older rough paths.

\begin{defn*}   
Let $\alpha\in (0,1]$, $\theta>\alpha$ and $X\in C^{\alpha}([0,T],\R^\ell)$. 
A \emph{path}
\begin{align*}
y\in C^{\alpha}([0,T],\R^d)
\end{align*}
 is said \emph{controlled by $X$}, with a remainder of order $\theta$, if there is a path 
\begin{align*}
y'\in  C^{\theta-\alpha}\big([0,T],L(\R^\ell,\R^d)\big),
\end{align*}
and  a map $R : [0,T]^2\rightarrow \R^d$, such that 
\begin{align*}
\norm{R}_{\theta} := \sup_{s,t\in [0,T], t\ne s}\frac{\norm{R^y_{s,t}}}{\abs{t-s}^\theta}<\infty.
\end{align*}
and
\begin{align*}
y_{st} = y'_sX_{st} + R_{st}.
\end{align*}
\end{defn*}

We denote by $\mathcal{D}^{\alpha,\theta}([0,T],\R^d)$ the space of such pairs $(y,y')$, and talk of $(y,y')$ as a controlled path. We make no reference to the reference path $X$ in the notation as there is no risk of confusion. We define a semi-norm on $\mathcal{D}^{\alpha,\theta}([0,T],\R^d)$ setting
\begin{align*}
\norm{(y,y')}_{\alpha,\theta} := \norm{y'}_{{\color{black} \theta-\alpha}} + \norm{R^y}_{\theta};
\end{align*}
a norm is defined by the formula
\begin{align*}
\norm{(y,y')}_{\alpha,\theta}^* := \norm{y_0}+\norm{y'_0} + \norm{(y,y')}_{\alpha,\theta}.
\end{align*}
The next proposition says that the class of controlled paths is stable by nonlinear maps. See Corollary 3 and Proposition 4 of Gubinelli's original statement \cite{Gubinelli}.

\begin{prop}   \label{lem:comp:DX}
For $(y,y')\in \mathcal{D}^{\alpha,\theta}([0,T],\R^d)$, and $f\in
C^{1,\epsilon}_b$ with $\epsilon\in (0,1]$, one has 
\begin{align*}
\big(f(y), f(y)'\big) := \big(f(y),Df(y)y'\big)\in \mathcal{D}^{\alpha,\theta'},
\end{align*}
with $\theta' := \min(\theta,\alpha(1+\epsilon))$, and 
\begin{equation}   \label{eq:comp:DX}
\begin{split}
\big\|(&f(y) , f(y)')\big\|_{\alpha,\theta'}  \lesssim_T \norm{f}_{C^{1,\epsilon}_b}\\
&\left(\norm{y_0'}+\norm{(y,y')}_{\alpha,\theta}+\left[\norm{y_0'}
+\norm{(y,y')}_{\alpha,\theta}\right]^{1+\epsilon}+\left[\norm{y_0'}+\norm{(y,y')}_{\alpha,\theta}\right]^{\theta'/\alpha}\right),
\end{split}\end{equation}
for an implicit multiplicative constant depending only on $T$, that decreases to $0$ when $T$ goes to $0$. 
\end{prop}

For a controlled path $(y,y')\in\mathcal{D}^{\alpha,\theta}([0,T],\R^d)$, with $\alpha+\theta>1$, it is elementary to see that the two-index map 
\begin{align*}
\mu_{st} := y_sX_{st} + y'_s\XX_{st},
\end{align*}
satisfies the estimate
\begin{align*}
\big|\mu_{st}-\mu_{su}-\mu_{ut}\big| \lesssim |t-s|^{\alpha+\theta},
\end{align*}
for all $0\leq s\leq u\leq t\leq T$, so one can use the sewing lemma to make sense of the integral
\begin{align*}
\int_0^\cdot y_sd{\bf X}_s
\end{align*}
as the additive functional associated to $\mu$. It satisfies, for any $s,t\in[0,T]$, the estimate
\begin{align}
\label{eq:int_rough}
\norm{\int_s^ty_rd{\rX}_r-y_sX_{st}-y'_s\XX_{st}} \lesssim \normh{\rX}\norm{(y,y')}_{\alpha,\theta}\abs{t-s}^{\alpha+\theta}.
\end{align}

\section*{Acknowledgement}
The second author of this article thanks the Center for Mathematical Modeling, Conicyt fund AFB 170001.

\bigskip
\bigskip

\noindent \textcolor{gray}{$\bullet$} {\sf I. Bailleul} -- {\small Institut de recherche mathématiques de Rennes, CNRS UMR 6625, Université de Rennes, 263 Avenue du General Leclerc
35042 RENNES, France.}

\hspace{1.6cm}{\it email:} ismael.bailleul@univ-rennes1.fr   \vspace{0.2cm}

\noindent \textcolor{gray}{$\bullet$} {\sf A. Brault} -- {\small  MAP5, CNRS UMR 8145, Université de Paris, Paris, France, and Center for Mathematical Modeling, UMI CNRS 2807, University of Chile, Chile.}

\hspace{1.6cm}{\it email:} abrault@dim.uchile.cl   \vspace{0.2cm}

\noindent \textcolor{gray}{$\bullet$} {\sf L. Coutin} -- {\small Institut de Mathématiques de Toulouse, CNRS UMR 5219, Université Paul Sabatier, 118 Route 	de Narbonne, 31062 Toulouse Cedex 9, France.}

\hspace{1.6cm}{\it email:} laure.coutin@math.univ-toulouse.fr

\end{document}